\documentclass[12pt]{amsart}

\scrollmode

\usepackage{amsfonts,amssymb,amsmath,amsthm}
\usepackage{mathrsfs}
\usepackage{mathtools}
\usepackage{url}
\usepackage{enumerate}

\usepackage{amsmath,amsfonts,amsthm,url,color,amssymb}
\usepackage{graphicx}

\newcommand{\realmathbb}{\mathbb}

\newtheorem{theorem}{Theorem}[section]
\newtheorem{lemma}[theorem]{Lemma}

\theoremstyle{definition}

\newtheorem{proposition}[theorem]{Proposition}

\newtheorem{corollary}[theorem]{Corollary}
\theoremstyle{remark}
\newtheorem{remark}[theorem]{Remark}

\numberwithin{equation}{section}

\begin{document}

\title[Contact Process under heavy-tailed renewals]{Contact Process under heavy-tailed renewals on finite graphs}


\author{Luiz Renato Fontes}
\address{
Instituto de  Matem\'{a}tica e estat\'{i}stica. Universidade de S\~{a}o Paulo, SP, Brazil.}
\curraddr{}
\email{lrfontes@usp.br}
\thanks{}

\author{Pablo Almeida Gomes}
\address{Instituto de Ci\^{e}ncias Exatas. Universidade Federal de Minas Gerais, MG, Brazil }
\curraddr{}
\email{pabloag@mat.ufmg.br}
\thanks{}

\author{Remy Sanchis}
\address{Instituto de Ci\^{e}ncias Exatas. Universidade Federal de Minas Gerais, MG, Brazil }
\curraddr{}
\email{rsanchis@mat.ufmg.br}
\thanks{}


\dedicatory{}

\commby{}



\begin{abstract}
{We investigate a non-Markovian analogue of the Harris contact process in a finite connected graph $G=(V,E)$: an individual is attached to each site $x \in V$, and it can be infected or healthy; the infection propagates to healthy neighbors just as in the usual contact process, according to independent exponential times with a fixed rate $\lambda>0$; however, the recovery times for an individual are given by the points of a renewal process attached to its timeline, whose waiting times have distribution $\mu$ such that $\mu(t,\infty) = t^{-\alpha}L(t)$, where $1/2 < \alpha < 1$ and $L(\cdot)$ is a slowly varying function; the renewal processes are assumed to be independent for different sites. We show that, starting with a single infected individual, if $|V| < 2 + (2\alpha -1)/[(1-\alpha)(2-\alpha)]$, then the infection does not survive for any $\lambda$; and if $|V| > 1/(1-\alpha)$, then, for every $\lambda$, the infection has positive probability to survive.}
\end{abstract}

\maketitle

\noindent AMS 2010 Subject Classifications: {60K35 ; 82B43}

\noindent Keywords and Phrases: {Contact process ; Percolation; Phase transition; Heavy tails}

\section{Introduction}

\noindent The contact process with renewal cures was the object of recent analyses in the literature; see \cite{FMMV2018} and \cite{FMV2018}.
Roughly speaking, there are two kinds of results, depending on the tail of the inter cure waiting time distribution, more precisely on the its index of polynomial decay --- let us call it $\alpha$. One kind is for $\alpha > 1$, stating in~\cite{FMV2018} that, under a certain monotonicity extra condition, the usual phase transition on the infection parameter, let us call it $\lambda$, is nontrivial, that is, there exists a nontrivial critical value $\lambda_c$ such that we have that the contact process on $\realmathbb{Z}$ started with a single infected individual dies out almost surely for $\lambda<\lambda_c$. The second kind of result holds for $\alpha<1$ and says in~\cite{FMMV2018} that, under certain fairly mild regularity conditions, $\lambda_c=0$, that is, we have survival of contact process with positive probability on any infinite graph and any $\lambda>0$. 

The present work looks again at the second situation above. The main factor behind the second kind of result is the occurrence with large probability of a \emph{tunneling event} at a large time, say T, where an infected individual located at the \emph{edge of the region currently explored by the infection} --- think of the underlying graph as $\realmathbb{Z}$, and the unexplored region as the one to the right of the infected individual in question --- sees a sequence of individuals whose current inter cure waiting times are of order $T$. This makes for larger and larger probability of spreading the infection to unexplored regions at larger and larger times, yielding positive probability of spreading it to all the space in the long run.

This picture suggests that we may not need infinite space to spread the infection forever, and that is the issue we analyse in the present article. We consider here a finite connected graph, and inter cure waiting time distributions in the basin of attraction of an $\alpha$-stable law, with $\alpha<1$, and give upper and lower bounds on critical size $k_c$ of the graph in terms of $\alpha$, above which we have survival as above, and below which the infection dies out for any $\lambda$. The bounds are quite sharp, leaving a gap of indetermination of size at most 1.

In comparison with the approach and conditions for the result of \cite{FMMV2018}, we believe that there is room for relaxing our sufficient conditions for survival for any $\lambda>0$, in the direction of the conditions of \cite{FMMV2018}, since the approaches are similar (but it is not so clear how an upper bound on $k_c$ would depend on $\alpha$). On the other hand, as will become clear below, our extinction result relies much more heavily on our regularity assumption on the
inter cure waiting time distribution.

\smallskip

The remainder of this article is organized in three more sections. We next describe the model, auxiliary results, and our bounds, which are collected in a single result, Theorem~\ref{th: principal}. The remaining two sections are devoted to our upper and lower bounds for $k_c$, one section for each, in this order.


\section{Preliminaries and Main Result}
\noindent Let us fix some notation and then define our model and describe our results.

We will consider versions of a renewal process $R = \{ S_n = T_1 + \ldots + T_n; ~n \in \realmathbb{N} \}$, with \textit{waiting times} $\{T_i\}_{i \in \realmathbb{N}}$ given as usual by i.i.d.~non-negative random variables. Let $U$ denote the associated \textit{renewal measure}, given, we recall, by $U(B) = \sum_{n \geq 1} P(S_n  \in B)$ for every Borel set $B \in \mathscr{B}(\realmathbb{R})$. For $t>0$, let $N(t) = \sup\{ n \in \realmathbb{N} ; S_n \leq t \}$ denote the number of renewals of $R$  up to time $t$. We also consider the \textit{current time} and \textit{excess time} at $t$ of $R$, given respectively by
\[ C(t) = t - S_{N(t)}\, \mbox{ and }\, E(t) = S_{N(t)+1} - t.\]

In this article, we will take the common probability distribution $\mu$ of the waiting times in the basin of attraction of an $\alpha$-stable law, that is, 
\begin{equation}\label{eq: stable}
\mu(t, \infty) = L(t)t^{-\alpha}, ~ t > 0,
\end{equation}     
where $L(\cdot)$ is a \textit{slowly varying function}, and $\alpha$ is a parameter, in principle, in $(0,1]$,
which in our context will be called \textit{cure index}. 

Let us recall two known results concerning renewal processes with such distribution, to be used below.

\begin{theorem}[Theorem 2 in \cite{E1970}]\label{th: ET}
	Let $\mu$ be as above with $1/2 < \alpha \leq 1$. Then, for every $h >0$, as $t \longrightarrow \infty$,
	\[
	U(t+h) - U(t) \sim \frac{C_{\alpha}h}{m(t)},
	\] 
	where $C_{\alpha} = [\Gamma(\alpha)\Gamma(2-\alpha)]^{-1}$ and $m(t) =  \int_0^{t} \mu(x,\infty)dx$.
\end{theorem}  

The second result is contained in the celebrated Dynkin-Lamperti Theorem, for which we refer again to \cite{E1970}
(paragraph right below (9.1), Section 9), or to \cite{F1971}, Chapter XIV.3.

\begin{theorem}\label{th: DL}
	Let $\mu$ be as above with $0 < \alpha < 1$. Then
	\[
	\lim_{t \rightarrow \infty} P\left(\frac{E(t)}{t} \leq x \right) = \int_0^{x} \frac{C_{\alpha}}{y^{\alpha}(y+1)}dy, ~\forall x>0,
	\] 
	where $C_{\alpha} = [\Gamma(\alpha)\Gamma(2-\alpha)]^{-1}$.
\end{theorem}

We now define the Renewal Contact Process (RCP), denoted by $(\zeta_t)_{t \geq 0}$. Given a connected graph $G=(V,E)$, a random variable $T$, a cure index $\alpha$ as above, and an infection rate $\lambda > 0$, we construct the RCP on $G$
graphically, \textit{\`a la} Harris, as follows: 

Let $T, \{ T_n^{x}\}_{x\in V, n \in \realmathbb{N}} $ be i.i.d.~random variables with distribution $\mu$ as in \eqref{eq: stable}, and let $\{X_n^{e}\}_{e\in E, n \in \realmathbb{N}} $ be i.i.d.~random variables with rate $\lambda$ exponential distribution, independently of $\{ T_n^{x}\}_{x\in V, n \in \realmathbb{N}} $.

For  $x \in V$, let $R_x$ denote the renewal process with marks given by $\{ S_n^{x} = T_1^x + \cdots + T_n^x ; ~n \in \realmathbb{N} \}$. In the rest of this paper, $R$ denotes any renewal process with the same distribution as $R_x$. Furthermore, for $e \in E$, $R_e$ denotes the rate $\lambda$ Poisson process given by $ \{ S_n^{e} = X_1^e + \cdots + X_n^e; ~ n \in \realmathbb{N} \}$. Throughout the text $E_x(\cdot)$, $C_x(\cdot)$, $E_e(\cdot)$ and  $C_e(\cdot)$, denotes the \textit{excess time} and \textit{current time} of the process $R_x$, $x \in V$, and $R_e$, $e \in E$, respectively. 

Given these processes, the RCP is constructed according to the usual recipe: if $s < t$ and $x, y \in V$, a path from $(x,s)$ to $(y,t)$ is a c\`{a}dl\`{a}g function on $[s,t]$ for which there exist times $t_0 = s < t_1 < \cdots < t_n = t$ and $x_0 = x, x_1, \ldots, x_{n-1} = y$ in $V$ such that assumes $x_i$ in $[t_i, t_{i+1})$, and

\begin{itemize}
	\item $\langle x_{i} , x_{i+1} \rangle \in E, ~ i = 0, \ldots, n-2$; 
	\item $E_{\langle x_{i} , x_{i+1} \rangle}(t_{i}) = t_{i+1} - t_{i}, ~ i = 0, \ldots, n-2$; 
	\item $E_{x_i}(t_i) > t_{i+1} - t_{i}, ~ i = 0, \ldots, n-1$.
\end{itemize}

We define now, for each $t \geq 0$, the function of the state of the individuals, $\xi_{t}: V \longrightarrow \{0,1\}$. The model starts with a single infected individual, i.e., for some $v_0\in V$, $\xi_{0}(x) = 1 \Longleftrightarrow x = v_0$, and for $t > 0$, $\xi_t(x) = 1 \Longleftrightarrow$ there exists a path from $(v_0,0)$ to $(x,t)$. We say that the individual $x \in V$ is infected at time $t$, if $\xi_t(x) = 1$, and healthy otherwise. In this case, the set of infected individuals at time $t$, is given by $\zeta_{t} = \{ x \in V ; ~ \xi_t(x) = 1\}$.

The main result of this paper is 

\begin{theorem}\label{th: principal}
	Given $1/2 < \alpha < 1$, for any random variable $T$ whose distribution is in the basin of attraction of an $\alpha$-stable law, and any finite connected graph $G=(V,E)$, the RCP $(\zeta_t)_{t \geq 0}$ is such that
	\begin{enumerate}
		\item $P(\zeta_{t} \neq \emptyset, \forall t > 0) = 0, \quad \mbox{if}\quad |V| < 2 + \frac{2\alpha -1}{(1-\alpha)(2-\alpha)}, \quad \quad \forall \lambda > 0$;
		
		\item $P(\zeta_{t} \neq \emptyset, \forall t > 0) > 0, \quad \mbox{if}\quad |V| > \frac{1}{1-\alpha}, \hspace{23mm} \forall \lambda > 0.$
	\end{enumerate}
\end{theorem}

\begin{remark}\ 
	\begin{enumerate}
		\item Note that the bounds in our theorem  are quite sharp; writing $V_{+}(\alpha) = 1/(1-\alpha)$ and $V_{-}(\alpha) = 2 + (2\alpha - 1)/[(1-\alpha)(2-\alpha)]$, we have that $V_{+} - V_{-} < 1$. Thus, if $[V_{-} , V_{+}] \cap \realmathbb{Z} = \emptyset$, then the model is well understood for every possible graph size $|V|$; otherwise, there is a single indeterminate case. 
		\item When $0 < \alpha \leq 1/2$, we have the trivial lower bound $k_c > 1$; assuming the regularity conditions of \cite{C2015}, we can use 
		Theorem 2.1 
		of that reference, instead of the present Theorem \ref{th: ET}, to analogously conclude that the same upper bound
		holds: we have survival for $0 < \alpha < 1/2$ if $|V| \geq 2$, and for $\alpha = 1/2$ if $|V| \geq 3$ .   
	\end{enumerate}
\end{remark}


\section{Survival}
\noindent In this section we prove the second item of the Theorem \ref{th: principal}. The idea of the proof consist in showing that there exists a sequence of polynomially increasing time intervals, such that, with positive probability the following events take place: in each such interval, there exists an individual free of cure marks; each interval intersects the next, and in this intersection there exists a sub-polynomially sized interval where all individuals get infected. So if there exists a single infected individual at the beginning of the sequence, and the above events occur, then the infection survives forever.

Given the graph $G=(V,E)$ and the random variable $T$ of Theorem \ref{th: principal}, and having fixed the infection rate $\lambda > 0$, we start by choosing two constants as functions of $\lambda$ and $G$  that will be used in this section. Since  $|V| > 1/(1-\alpha)$ we can choose $\epsilon > 0$ in such way that $\beta := |V|(1-\alpha - 3\epsilon) > 1$. And since the graph $G = (V,E)$ is connected, there exist $l\geq1$ and a \textit{spanning path} $\tau = (e_1, e_2, \cdots, e_l)$, where $e_i = \langle v_{i-1}, v_i \rangle  \in E$, $i = 1, \cdots, l$, is an edge of $G$, with the following property: for each pair of vertices $(x,y) \in V^{2}$, $\tau$ has a sub-path $\tau(x,y) = (e_{i}, e_{i+1}, \cdots, e_{i+j})$, with $v_{i - 1} = x$ and $v_{i+j} = y$ for some $1 \leq i \leq l $ and $j \geq 0$. We note that there is a bound on $l$ in terms of $|V|$, namely $l\le 2|V|$. As a function of $\lambda$ and $l$, we choose $\gamma > \max\{1,l/\lambda\}$. From now on, $\epsilon$ and $\gamma$ are fixed.

With the objective to estimate the probability of existence of intervals without marks of the renewal process $R$, we derive the following corollary of Theorem \ref{th: ET}.

\begin{proposition}
	There exists $\hat{t}_1 > 0$ such that for all $t > \hat{t}_1$ 
	\[P(E(t) \leq  1) \leq \frac{1}{t^{1-\alpha - \epsilon}}.\]
\end{proposition}
\begin{proof} 
	First note that, given $t>0$, $U(t+1) - U(t) = \sum_{n \geq 1} P(S_n \in (t,t+1] )$,
	and let $M_t = \{n \geq 1 ~;~ S_n \in (t,t+1]\}$ be the number of renewal marks of $R$ in the interval $(t, t+1]$; then we have $U(t+1) - U(t) = \mathbb{E}(M_t)$. So,
	$P(E(t) \leq 1) = P(M_t \geq 1) \leq \mathbb{E}(M_t) =  U(t+1) - U(t)$.
	Since $L(\cdot)$ is slowly varying, we find $t_0$ such that $L(t) \geq t^{-{\epsilon}/{2}}$ for all $t> t_0$. Thus, making $h=1$ in Theorem \ref{th: ET}, we get that
	\[
		U(t+1) - U(t) \sim \frac{C_{\alpha}}{\int_0^t L(x)x^{-\alpha}dx} \leq \frac{C_{\alpha}}{\int_0^{t_0} L(x)x^{-\alpha}dx + \int_{t_0}^t 1/x^{\alpha + \frac{\epsilon}{2}  }dx },
	\]
	and thus may conclude that the left hand side is bounded above by $1/t^{1-\alpha - \epsilon}$  for all $t$ sufficiently large.
\end{proof}

Noticing that if $E(t) \in (s , s+1] $, then necessarily $E(t+s) \leq 1$, we have the following corollary to the above proposition.

\begin{corollary}\label{cr: corolario1}
	For all $m \in \realmathbb{N}$ and for all $t > \hat{t}_1$, we have $P(E(t) \leq m) \leq m/t^{1-\alpha - \epsilon}$.
\end{corollary}
\begin{proof}
	It is enough to observe that
	\begin{eqnarray*}
		P(E(t) \leq m) &=& \sum_{i=0}^{m-1} P(i < E(t) \leq i+1)\\
		&\leq& \sum_{i=0}^{m-1} P(E(t+ i) \leq 1)\\
		&\leq& \sum_{i=0}^{m-1} \frac{1}{(t+ i)^{1 -\alpha -\epsilon}} \leq \frac{m}{t^{1-\alpha - \epsilon}}.
	\end{eqnarray*}
\end{proof}

We will use Corollary  \ref{cr: corolario1} to show that, with high probability, certain intervals with polynomially growing size are free of cure events. 
For each $n \in \realmathbb{N}$, let $b_n = \gamma \log(n)$ and $c_n = \lceil b_n^{|V|(\alpha + \epsilon)+1} \rceil$. It follows that there exists $n_0$, such that $c_nb_n < n^{\epsilon}/2$, $\forall n \geq n_0$. Then, for each $n \geq n_0$, we define
\[
t_n = \hat{t}_1 + \sum_{j=n_0}^{n} [j^{\epsilon} - c_jb_j].
\]
It follows that $t_n \geq \sum_{j=n_0}^{n} j^{\epsilon}/2$, hence, for all $n$ large enough  $t_n > n$. 

Consider now the event $A_n = \{ \exists x \in V ; ~ E_x (t_n) > (n+1)^{\epsilon}  \}$. In this event, at least one of the individuals has no cure during the interval $(t_n, t_n + (n+1)^{\epsilon})$. The next proposition gives a lower bound for the probability of occurrence of this event.

\begin{proposition} \label{pr: k}
	There exists $n_1 \in \realmathbb{N}$, such that, for $n > n_1$, we have $ P(A_n^{c}) \leq {1}/{n^\beta} $, where $\beta = |V|(1-\alpha - 3\epsilon) > 1$. 
\end{proposition}
\begin{proof}
	Let us take $n$ large enough and $t_n  > n$ so that we may apply Corollary \ref{cr: corolario1} to get
	\begin{eqnarray*}
		P(A_n^{c}) = P(E_x(t_n) \leq (n+1)^{\epsilon}, \forall x \in V)
		&\leq& \left( \frac{\lceil (n+1)^{\epsilon} \rceil}{t_n^{1-\alpha - \epsilon}}\right)^{|V|}\\ & \leq& \left( \frac{n^{2\epsilon}}{n^{(1-\alpha - \epsilon)}}\right)^{|V|}\\
		&= & \frac{1}{n^{|V|(1-\alpha - 3\epsilon) }}\\
		&=& \frac{1}{n^\beta}.
	\end{eqnarray*}
\end{proof}

The next step is to show that, with high probability, at least one of the following $c_n$ intervals with size  $b_n$, is free of all cure processes $R_x$, $x \in V$. We begin with the following lemma:

\begin{lemma}\label{lm: lema1}
	There exists $\hat{t}_2 > 0$, such that, if $t > \hat{t}_2$, then, for all $s> 0$, we have $P(T>s+t | T > s) \geq 1/t^{\alpha + \epsilon}$.
\end{lemma}
\begin{proof}
	We start with the case $s \leq t$, where there exists $t^{\ast}$ such that
	\[
	P( T > t+s  | T>s )
	\geq P(T > t+s) \geq P(T > 2t)
	= \frac{L(2t)}{(2t)^{\alpha}} \geq \frac{1}{t^{\alpha + \epsilon}},
	\]
	for all $t > t^{\ast}$. For the other case, namely $s>t$, we have that
	\[
		P( T > t+s  | T>s ) = \frac{P(T > t+s)}{P(T > s)}
		\geq \frac{P(T>2s)}{P(T>s)} = \frac{L(2s)}{L(s)}\left( \frac{1}{2}\right)^{\alpha}.
	\]
	Since $L(\cdot)$ is slowly varying, $L(2s)/L(s) \to 1$ as $s\to\infty$. It follows that there exists $s^{\ast}$ such that, if $s > t > s^{\ast}$, then
	\[P(T>t+s | T>s) \geq \frac{L(2s)}{L(s)}\left( \frac{1}{2}\right)^{\alpha} \geq  \frac{1}{(s^{\ast})^{\alpha+\epsilon}} > \frac{1}{t^{\alpha + \epsilon}}.
	\]
	To conclude the proof, take $\hat{t}_2 = \max\{t^{\ast} , s^{\ast}\}.$
\end{proof}

Let  $t_0>0$ be fixed, and consider the sub-$\sigma$ algebra $\mathcal{F}_{t_0}$ of the underlying $\sigma$ algebra of the model consisting of renewal events taking place up to time $t_0$. We have the following lemma.

\begin{lemma}\label{lema2}
	Given $t_0 > 0$, then, for all $t > \hat{t}_2$ and all $x \in V$, almost surely
	\[
	P\left(E_x(t_0) > t ~ \Big\vert ~ \mathcal{F}_{t_0}\right) \geq \frac{1}{t^{\alpha + \epsilon}}.
	\]
\end{lemma}
\begin{proof}
	Almost surely
	\begin{eqnarray*}
		P\left( E_x(t_0) > t ~ \Big\vert ~ \mathcal{F}_{t_0}\right)(\omega)\!\! &=& \!\! P\left( T > t + \left(t_0 - S^{x}_{N(t_0)}(\omega)\right) ~ \Big\vert ~ T > t_0 - S^{x}_{N(t_0)}(\omega)  \right) \\
		&\geq&\!\!  \frac{1}{t^{\alpha + \epsilon}},
	\end{eqnarray*}
	where we used Lemma \ref{lm: lema1} in the last passage.
\end{proof}

For $n > n_0$, we define $B_n = \{ \exists j \in [0, c_n) \cap \realmathbb{Z} ; ~ E_x(t_n + j b_n ) > b_n, \forall x \in V\}$. Observe that, on the occurrence of  $B_n$ it is assured that at least one of the  $c_n$ intervals of size $b_n$ has no event of cure. Using the lemma above, we get an upper bound for the probability of $B_n$.

\begin{proposition}\label{pr: B}
	Let $n_2 = \inf\{n>n_0  ; \ b_n > \hat{t}_2\}$. If $n> n_2$, then $P(B_n^{c}) \leq 1/n^{\gamma}$.
\end{proposition}
\begin{proof}
	For simplicity, we write $C_{n,j} = \{ \exists x \in V ; ~ E_x(t_n + jb_n) \leq b_n\}$. Then we have
	\begin{eqnarray*}
		P(B_n^{c}) &=& P (C_{n,j}, \forall j= 0, \ldots, c_n - 1)\\ 
		&=& \displaystyle \prod_{j=0}^{c_n -1} P\left(C_{n,j} ~ \Bigg\vert ~ \bigcap_{i=0}^{j-1} C_{n,i} \right)\\
		&=& \prod_{j=0}^{c_n - 1}\left[ 1 - P\left( E_x(t_n+jb_n) > b_n, \forall x \in V ~ \Bigg\vert ~ \bigcap_{i=0}^{j-1} C_{n,i} \right) \right] .
	\end{eqnarray*}
	Since the events $C_{n,i}$, where $0 \leq i < j$, occur before $t_n + jb_n$, using the Lemma \ref{lema2}, we have $P(E(t_n+jb_n) > b_n | \cap_{i=0}^{j-1} C_{n,i} ) \geq 1/b_n^{\alpha + \epsilon}$. Hence,
	\begin{eqnarray*}
		P(B_n^{c}) &=& \prod_{j=0}^{c_n - 1}\left[ 1 - P \left( E(t_n+jb_n) > b_n ~ \Bigg\vert ~ \bigcap_{i=0}^{j-1} C_{n,i} \right)^{|V|} \right]\\ 
		&\leq& \prod_{j=0}^{c_n - 1} \left( 1 - \frac{1}{b_n^{|V|(\alpha + \epsilon)}} \right) \\
		&=& \left( 1 - \frac{1}{b_n^{|V|(\alpha + \epsilon)}} \right)^{c_n} \leq e^{{-c_n}/{b_n^{|V|(\alpha + \epsilon)}}} \leq e^{-b_n} = \frac{1}{n^{\gamma}}. 
	\end{eqnarray*}
\end{proof}

We use the memory loss of the exponential distribution to show that, with positive probability, in each one of the $c_n$ intervals, we have the occurrence of a \textit{stairway of infection}. Given $t>0$, and recalling the spanning path $\tau = (e_1, \ldots, e_l)$, presented above, on the second paragraph of this section, we define the following random variables: 
\[
Y_i^{t} = 
\begin{cases} 
         t, & \mbox{if }   i = 0 \\ 
         Y_{i-1}^{t} + E_{e_i}(Y_{i-1}^{t}), & \mbox{if } 1 \leq i \leq l.
\end{cases}
\]
Observe that, since from every pair $(x,y) \in V^2$, $\tau$ has a sub-path starting at $x$ and ending at $y$, if at time $t$ there is at least one infected individual in $V$, then, whenever $E_x(t) > Y_{l}^t - t$ for all $x \in V$, we will have that all individuals are infected at time $Y_{l}^t$. 

\begin{proposition}\label{pr: escada}
	Given $m > n_2 \in \realmathbb{N}$, the event 
	\[
	 \displaystyle  C_m := \left\{\bigcap_{n \geq  m} \bigcap_{j=0}^{c_n - 1} \left\{ Y_{l}^{t_n +jb_n} - (t_n + jb_n)  \leq b_n \right\}  \right\}
	\] 
	is such that $P(C_m) >0$.
\end{proposition}
\begin{proof}
	Observe that, due to the memory loss of the possibly infection distribution,  for all $t$, the random variables $Y_{i}^t - Y_{i-1}^t$, $i = 1, 2, \ldots, l$, are i.i.d.~ exponentially distributed with rate  $\lambda$. Hence, 
	\begin{eqnarray*}
		P\left( Y_{l}^{t} - t \leq b_n\right) = P \left( \sum_{i=1}^{l} (Y_{i}^t - Y_{i-1}^t) \leq b_n \right)
		&\geq& P \left(\max_{1 \leq i \leq l } (Y_{i}^t - Y_{i-1}^t) \leq \frac{b_n}{l} \right)\\
		&=& \left(1 - e^{\frac{-\lambda b_n}{l}}\right)^{l}.
	\end{eqnarray*}
	It readily follows that 
	\[  
	\displaystyle P \left(\bigcap_{n \geq m} \bigcap_{j=0}^{ c_n -1 } \left\{ Y_{l}^{t_n +jb_n} - (t_n + jb_n)  \leq b_n \right\}  \right) \geq  \prod_{n \geq m}  \left(1 - e^{\frac{-\lambda b_n}{l}}\right)^{lc_n },
	\]
	and since $b_n = \gamma\log(n)$, taking logarithms, we obtain that, for some constant $c > 0$,
	\begin{eqnarray*}
		\log \left( \prod_{n \geq m}  \left(1 - e^{\frac{-\lambda b_n}{l}}\right)^{lc_n } \right) &>& -cl\sum_{n \geq m} c_n  e^{\frac{-\lambda b_n}{l}} \\ &=& -cl\sum_{n \geq m}  c_n {n^{-\frac{\gamma\lambda}{l}}}.
	\end{eqnarray*}
	Finally, since $\gamma$ was chosen in such way that $\gamma \lambda > l$, and $c_n = \lceil (b_n)^{|V|(\alpha+\epsilon)+1} \rceil$, the latter sum in convergent and thus, the product above is positive.
\end{proof}

\subsection{Proof of item (2) of Theorem \ref{th: principal}}
\noindent Using the propositions above we can conclude the proof of Theorem \ref{th: principal} (2). Let us start with some definitions. For each $t>0$, we say that a configuration $\omega \in \Omega$ is $t-\text{bad}$, if there exist $s \geq t$ and $\{n_x \in \realmathbb{N}, x \in V\}$, such that $S_{n_x}^x =s,$ for all $x \in V$. This means that there is an instant $s$ after or equal to $t$, where each individual of $V$ simultaneously gets a cure mark, each one of his respective cure process $R_x$. We say that $\omega$ is bad if it is $0-\text{bad}$, and is good otherwise.

Let $n_3 = n_1 \vee n_2$, where $n_1$ and $n_2$ are given in Proposition \ref{pr: k} and Proposition \ref{pr: B}, respectively. Given $m > n_3 \in \realmathbb{N}$, we define $\tilde{A}_m = \cap_{n \geq m} A_n$. Since $\epsilon > 0$ was chosen in such way that $\beta = |V|(1-\alpha -3\epsilon) > 1$, it follows from Proposition \ref{pr: k} and the union bound that $P(\tilde{A}_m^{c}) \to 0$ as $m$ goes to infinity.
Since $\{\omega \text{ is t-bad}, \forall t>0  \} \cap \tilde{A}_m = \emptyset$, if we suppose that $P(\omega \text{ is bad}) = 1$, then by the strong Markov property of our system we have that $P(\omega \text{ is t-bad}, \forall t>0) = 1$, which in turn implies that $P(\tilde{A}_m) = 0$, in contradiction with what we just argued. Thus, we have that $P(\omega \text{ is good}) = p > 0$.

Recalling the event $C_{n,j}$ defined in the proof of Proposition \ref{pr: B}, we have that $P(C_{n,1})$ goes to $0$ as $n$ goes to infinity. Now, we define $\tilde{B}_m = C_{m,1}^{C} \cap \left(\cap_{n > m} B_n \right) $. Remembering that $\gamma > 1$, applying again the union bound and Propositions \ref{pr: k} and \ref{pr: B}, we obtain
\begin{eqnarray}\label{eq: end}
\nonumber  1 - P(\{\omega \text{ is good}\} \cap \tilde{A}_m \cap \tilde{B}_m) &\leq& 
P(\{\omega \text{ is bad}\})  + \sum_{n \geq m} P(A_n^c ) + P(C_{m,1}) +  \sum_{n > m} P(B_n^c)  \\
&\leq& (1- p) + \sum_{n \geq m} \frac{1}{n^{\beta}} + P(C_{m,1}) +  \sum_{n > m} \frac{1}{n^{\gamma}} < 1,  
\end{eqnarray}
for $m$ large. We fix now $m > n_3 \in \realmathbb{N}$ satisfying \eqref{eq: end}.  

Now see that, if at time $t_m$ there exists a infected individual, and the events $\tilde{A}_m$, $\tilde{B}_m$ and $C_m$ occur simultaneously, then the infection survives forever. That is,
\[
\{\zeta_{t} \neq \emptyset, \forall t > 0\} \supset \{\zeta_{t_m} \neq \emptyset\} \cap \tilde{A}_m \cap \tilde{B}_m \cap C_m .
\]

Follow from the independence between the cure and infection process, that our probability measure is given by $P = P_1 \times P_2$, where $P_1$ and $P_2$ are the marginal probabilities of the cure process and infection process respectively. Analogously, let $\Omega = \Omega_1 \times \Omega_2$. Since the event $\{\omega \text{ is good}\} \cap \tilde{A}_m \cap \tilde{B}_m$, only depends of cure process, we can write  $\{\omega \text{ is good}\} \cap \tilde{A}_m \cap \tilde{B}_m = \Lambda \times \Omega_2$. Thus, from \eqref{eq: end},
\[
P_1(\Lambda) = P(\{\omega \text{ is good}\} \cap \tilde{A}_m \cap \tilde{B}_m) > 0.
\]

Given the independence between the cure and infection processes, it is enough to argue that in the event 
$\{\omega \text{ is good}\}$, we have that $\zeta_{s} \neq \emptyset$ with positive probability for any $s$. But this should be quite clear since in that event, which depends solely on the renewal processes, there are no obstacles for the spread of the infection in any finite time.

Wrapping up, we may write
\begin{multline*}
P(\{\zeta_{t} \neq \emptyset, \forall t > 0\}) \geq
P(  \{\zeta_{t_m} \neq \emptyset\} \cap \tilde{A}_m \cap \tilde{B}_m \cap C_m \cap  \{\omega \text{ is good}\}) 
= \\  
P( \zeta_{t_m} \neq \emptyset ~|~   \tilde{A}_m \cap \tilde{B}_m \cap \{\omega \text{ is good}\} )\,
P( \tilde{A}_m \cap \tilde{B}_m \cap \{\omega \text{ is good}\} )\,
P(C_m) = \\
\int_{\Lambda} P_2(\zeta_{t_m}(\omega_1,\omega_2) \neq \emptyset) ~|~ \omega_1) dP_1(\omega_1) \, P_1(\Lambda) \, P(C_m) > 0.
\end{multline*}


\section{Extinction}
\noindent In this section we prove the first item of the Theorem \ref{th: principal}. The idea consists in creating a sequence of disjoint random time intervals which, for the infection to survive, 
would be required to contain at least one mark of any of the cure processes.
We then resort to a domination argument to show that we may find a subsequence of those intervals with bounded lengths, and the result readily follows from that. 

Given $G=(V,E)$ and $T$ as in Theorem \ref{th: principal}, we start defining time intervals $(S_n, S_{n+1}]$. For this, recalling that $v_0 \in V$ is the single one initially infected individual, for each individual $x \in V$, let
\[
\displaystyle 
X_{1,x}=\left\{\begin{array}{rc}
T_1^{v_0},&\mbox{se}\quad x = v_0,\\
0, &\mbox{se}\quad x \neq v_0.
\end{array}\right.
\]
Set $S_1 = X_1 = \max\{X_{1,x} ~;~ x \in V\}$. Again, for each individual $x \in V$, let $W_{1,x} = X_1 - X_{1,x}$. And define, $x_1 = \arg\max\{ X_{1,x}  ~; ~ x \in V\}$.

Given $t^{\ast} > 0$, for a given $n \in \realmathbb{N}$, we assume defined $X_{m,x}$, $W_{m,x}$, $X_m$, $S_m$, $x_m$, $m=1, \ldots n$, $x \in V$, and set 
\[
\displaystyle 
X_{n+1,x}=\begin{cases}
0, &\mbox{if}\quad x = x_n,\\
E_x(S_n), &\mbox{if}\quad x \neq x_n \quad \textrm{and} \quad W_{n,x} \geq t^{\ast},\\
E_x(S_n + t^{\ast}) + t^{\ast}, &\mbox{if}\quad x \neq x_n \quad \textrm{and} \quad W_{n,x} < t^{\ast}.
\end{cases}
\]
Analogously, we define $X_{n+1} = \max\{X_{n+1,x} ~; ~ x \in V\}$, $S_{n+1} = S_n + X_{n+1}$, for each individual $x \in V$, $W_{n+1, x} = X_{n+1} - X_{n+1,x}$, and also set $x_{n+1} = \arg\max\{ X_{n+1,x} ~ ; ~ x \in V\}$. 

The conditional distribution of $X_{n+1,x}$ on the past is given by
\begin{multline}\label{eq: Markov}
\displaystyle 
X_{n+1,x} ~|~ X_{m,x},~ 1\leq m \leq n,~ x \in V \\
\sim \begin{cases}
0, &\mbox{if}\quad x = x_n,\\ 
E(W_{n,x}),&\mbox{if}\quad x \neq x_n \quad \textrm{and} \quad W_{n,x} \geq t^{\ast},\\
E(W_{n,x}+t^{\ast}) + t^{\ast}, &\mbox{if}\quad x \neq x_n \quad \textrm{and} \quad W_{n,x} < t^{\ast},
\end{cases}
\end{multline}
where, we recall, $E(\cdot)$ denotes the \textit{excess time} of a renewal process $R$.

A necessary condition for the infection to survive is that in each one of the time intervals $(S_n, S_{n+1}]$, there is at least one mark of some infection process $R_{e}$, $e \in E$. It readily follows that
\begin{equation}\label{eq: continencia}
\displaystyle P \left( \zeta_t \neq \emptyset, \forall t>0 ~ \bigg\vert ~ \left\{\lim_{n \rightarrow \infty} X_n = \infty\right\}^{c} \right) = 0.
\end{equation}

We will show below that $P(\lim_{n \rightarrow \infty} X_n = \infty)=0$ by resorting to a domination argument.

\subsection{Domination} 
\noindent We will control the behavior of the random variables $(X_n)_{n \in \realmathbb{N}}$ through Theorem \ref{th: DL} and two technical propositions, as follows.  

\begin{proposition}\label{pr: Calda}
	Given $0 < \eta < 1$, there exists $t_{\eta} > 0$ such that 
	\[
	P\left( \frac{E(t)}{t} > e^n \right) < \left( \frac{1+\eta}{e^{\alpha}}\right)^{n},  ~ \forall n \in \realmathbb{N}, ~ \forall t > t_{\eta}.
	\]
\end{proposition}
\begin{proof}
	We claim that there exists $t_{\eta} > 0$ such that 
	\[
	(1-\eta)^n < L(e^n t)/L(t) < (1+\eta)^n , ~ \forall n \in \realmathbb{N}, ~ \forall t > t_{\eta}.
	\]
	Indeed, since $L$ is slowly-varying, we have that $\lim_{t \rightarrow \infty} L(et)/L(t) = 1$; thus, there exists $t_{\eta}$ where the claim is true for $n=1$ and $t > t_{\eta}$. Let $s = e^n t$, and write
	\[
	\frac{L(e^{n+1} t)}{L(t)} = \frac{L(e^{n+1} t)}{L(e^n t)}\frac{L(e^{n} t)}{L(t)} 
	= \frac{L(e s)}{L(s)}\frac{L(e^{n} t)}{L(t)}.
	\]
	Since $s > t > t_{\eta}$, and supposing the claim is true for $t > t_{\eta}$ and a given $n \in \realmathbb{N}$, then we have that the same is true for $n+1$, and the claim follows by induction.
	
	Fixing $t > 0$, and conditioning on the variable $C(t) = t - S_{N(t)}$, whose distribution function we denote by $F_t$, we have that
	\begin{eqnarray*}
		P\left(E(t) > e^n t\right) &=& \int_{0}^t P\left(E(t) > e^n t ~|~ C(t)=s \right) dF_t(s)\\
		&=& \int_{0}^t P\left(T > e^n t + s ~|~ T > s \right)  dF_t(s)\\
		&=& \int_0^t\frac{P(T> e^n t + s) }{P(T > s)}   dF_t(s)\\
		&\leq&  \int_0^t\frac{P(T> e^n t) }{P(T > t)}   dF_t(s)\\
		&=& \frac{L(e^n t)}{(e^n t)^{\alpha}} \div \frac{L(t)}{t^{\alpha}} \int_0^t  dF_t(s)\\
		&=& \frac{L(e^n t)}{L(t)}\frac{1}{e^{\alpha n}}.
	\end{eqnarray*}
	Hence, with the same $t_{\eta}$, the result follows directly from the claim above.
\end{proof}

Recalling the constant $C_{\alpha}$ in the Theorem \ref{th: ET} and Theorem \ref{th: DL}, let $M = |V| - 1$ and consider $Y_1, \ldots, Y_M$, independent random variables  with common density
\begin{equation}\label{eq: densidade}
\displaystyle 
f(y)=\left\{\begin{array}{rc}
0,&\mbox{if}\quad y\leq0,\\
\frac{C_{\alpha}}{y^{\alpha}(1+y)}, &\mbox{if}\quad y > 0.
\end{array}\right.
\end{equation}
And let $Y$ be the random variable  
\begin{equation}\label{eq: Y}
Y \equiv \max\{Y_i ~ ; ~ i=1, \ldots, M\}.
\end{equation}

Since $M < 1 + (2\alpha - 1)/[(1-\alpha)(2-\alpha)]$, we have $\mathbb{E}[\log(Y)] <0$. A proof of this fact can be found in the Appendix. Note that, $\mathbb{E}[Y^t] < \infty \Leftrightarrow \mathbb{E}[Y_1^{t}] < \infty \Leftrightarrow t \in (-(1-\alpha), \alpha)$. Let $\Phi: (-(1-\alpha), \alpha) \to \realmathbb{R}$ be defined by 
\[
\Phi(t) = \mathbb{E}[e^{t \log(Y)}] = \mathbb{E}[Y^{t}].
\]
We observe that $\Phi$ is differentiable at $0$, with $\Phi'(0) = \mathbb{E}[\log(Y)] < 0$ and $\Phi(0) =1$. Hence, there exists $0 < \theta < \alpha$, with $\Phi(\theta) < 1$, that is $\mathbb{E}[Y^{\theta}] < 1$.

Let  $N \in \realmathbb{N}$ be such that $\log(N) \in \realmathbb{N}$, and consider  $a_j = j/N$, if $j=0, \ldots, N^2$, and $a_j = N\exp(j - N^2)$, if $j > N^2$. For each $j \in \realmathbb{N}$, let $I_j = (a_{j-1}, a_j]$, and consider the following truncation of $Y$:
\[
\overline{Y}_N = \sum_{j=1}^{N^2} a_j \mathbb{1}_{\{Y \in I_j\}}.
\]
For given $\mu$ such that $\mathbb{E}[Y^{\theta}] < \mu < 1$, it follows by dominated convergence that, for $N$ sufficiently large, 
\begin{equation}\label{eq: mu}
\mathbb{E}[ (\overline{Y}_N)^{\theta}] < \mu < 1.
\end{equation}

Let $a = (1+\eta)/e^{\alpha}$. From now on, we fix $0 < \eta < 1$ so that $ae^{\theta} < 1$. Given $\rho > 0$, for each $j \in \realmathbb{N}$, we define
\begin{equation}\label{eq: prho}
\displaystyle 
p_{N, \rho, j}=\left\{\begin{array}{rc}
P(Y \in I_j) + \rho ,&\mbox{if}\quad j \leq N^2,\\
Ma^{\log(N) + j - N^2 - 2}, &\mbox{if}\quad j > N^2.
\end{array}\right.
\end{equation}
We also define $C_{N,\rho} = \sum_{j \geq 1} p_{N,\rho, j}$. Recalling that \eqref{eq: mu} holds for $N$ sufficiently large, and since $ae^{\theta} <1$, then $N$ and $\rho$ can be chosen in such way that the following inequality is true
\begin{eqnarray}\label{eq: Cond2}
\nonumber &&\displaystyle \frac{1}{C_{N,\rho}} \left[ \sum_{j=1}^{N^2} \left[ a_j^{\theta} \left( P(Y \in I_j) + \rho \right) \right]  +  \sum_{j > N^2} a_j^{\theta} Ma^{\log(N) + j - N^2 - 2} \right] \\
\nonumber &=& \displaystyle \frac{1}{C_{N,\rho}} \left[ \mathbb{E}\left[(\overline{Y}_N)^{\theta}\right] + \rho  \sum_{j=1}^{N^2} a_j^{\theta} +  M \sum_{n \geq \log(N)-1} e^{\theta(n+2)}a^{n} 
\right] \\
&\leq& \displaystyle \frac{\mu}{C_{N,\rho}}.
\end{eqnarray}
In the following, $N$ and $\rho$ are fixed and satisfy Inequality \eqref{eq: Cond2}. In this case, we denote $C_{N,\rho}$ simply by $C$.

We define an auxiliary probability space, $\left( [0, \infty), \mathcal{F}, \mathbb{P} \right)$, where $\mathcal{F} = \sigma(I_j ~; ~ j \in \realmathbb{N})$ and for each $j \in \realmathbb{N}$, $\mathbb{P}(I_j) = p_j$, where 
\begin{equation}\label{eq: pj}
p_j = \frac{p_{N,\rho,j}}{C}.
\end{equation}
Let $\Tilde{Y}: [0,\infty) \longrightarrow (0, \infty)$, be a random variable in this space given by 
\begin{equation}\label{eq: dominad}
\Tilde{Y} = \sum_{j\geq 1} a_j \mathbb{1}_{I_j}.
\end{equation}
It follows directly from \eqref{eq: prho}, \eqref{eq: Cond2} and \eqref{eq: pj}, that $\tilde{Y}$ satisfies $\mathbb{E}[\Tilde{Y}^{\theta}] < \mu/C$.

We now apply Theorem \ref{th: DL} and Proposition \ref{pr: Calda} to establish our second technical proposition.

\begin{proposition}\label{pr: dominacao}
	There exists $t^{\ast} > 0$ and $ 1/2 > \delta > 0$ such that, for each $t_1, \ldots, t_M > t^{\ast}$, whenever $V_1,\ldots, V_M$ are independent random variables with marginal distributions such that for $i=1,\ldots,M$
	\[
	\quad\mbox{either} \quad  V_i \sim \frac{E(t_i)}{t_i} \quad \mbox{or} \quad V_i \sim \frac{E(t_i)}{t_i} + \delta  ,
	\]
	and $V \equiv \max\{V_i ~;~ i =1, \ldots, M \}$, then   $P\left( V \in I_j\right) < Cp_j, ~\forall j \in \realmathbb{N}$.
\end{proposition}
\begin{proof}
	If $1 \leq j \leq N^2$, then we use Theorem \ref{th: DL} to obtain  
	\[ 
	\displaystyle \lim_{t_1, \ldots t_M \to \infty} P\left(\max_{1 \leq i \leq M}\frac{E(t_i)}{t_i} \in I_j\right) = P( Y \in I_j) <  P( Y \in I_j) + \rho = Cp_j.
	\]
	Using the continuity of the limiting distribution of $E(t)/t$ as $t\to\infty$, it follows that, for $t_1, \ldots, t_M$ large enough and $\delta$ small enough, $P(V \in I_j) < Cp_j, ~\forall j\leq N^2$.
	
	Recalling that $a_j=e^{\log(N) + j - N^2}$ for all $j\geq N^2$, it follows from Proposition \ref{pr: Calda} that if $t> t_{\eta}$, then for all $j > N^2$ we have that 
	\[
	P\left( \frac{E(t)}{t} > a_{j-2}\right) < a^{\log(N) + j - N^2 - 2}.
	\]
	Observe that for all $j > N^2$ and $\delta$ small enough, we have for all possible cases of the marginal distributions of $V_i,\,i=1,\ldots, M$ that
	\begin{eqnarray*}
		P(V\in I_j)&\leq &P(V_i>a_{j-1}\mbox{ for some }i=1,\ldots, M)\\
		&\leq &P\left(\frac{E(t_i)}{t_i} > a_{j-2}\mbox{ for some }i=1,\ldots, M \right)\\
		&\leq &Ma^{\log(N) + j - N^2 - 2} = Cp_j.
	\end{eqnarray*}
\end{proof}

In the next proposition, we finally obtain the above mentioned domination. Let $\{\Tilde{Y}_m\}_{m \in \realmathbb{N}}$ be i.i.d. random variables  with the same distribution as $\Tilde{Y}$ in \eqref{eq: dominad}. 

\begin{proposition}\label{pr: domination}
	Let $\Tilde{t} = t^{\ast}/\delta$. Then, for every $n_0, m  \in \realmathbb{N}$,  
	\[ 
	\displaystyle P\left( X_{n_0} > \Tilde{t}, ~ X_{n_0+1} \geq X_{n_0}, ~ \ldots, ~ X_{n_0+m} \geq X_{n_0}\right) \leq C^{m} \mathbb{P} \left( \prod_{l=1}^m \Tilde{Y_l} \geq 1 \right). 
	\]
\end{proposition}
\begin{proof}
	For each $n \in \realmathbb{N}$ and $x \in V$, we define
	\begin{equation}\label{eq: Z}
	Z_{n+1,x} = \begin{cases} 
	\displaystyle \frac{X_{n+1,x}}{W_{n,x}}, & \mbox{if} \quad W_{n,x} \geq t^{\ast},\\ 
	\displaystyle \frac{X_{n+1,x} - t^{\ast}}{W_{n,x}+t^{\ast}} + \frac{t^{\ast}}{\Tilde{t}}, & \mbox{if} \quad  W_{n,x} < t^{\ast}.
	\end{cases}
	\end{equation}
	
	We set $Z_{n+1} = \max \{Z_{n+1,x} ~;~ x \in V \}$. Since $\delta < 1/2$ --- see Proposition \ref{pr: dominacao} ---, we have $2t^{\ast} < \Tilde{t} $. Notice also that for each $x \in V$, we have $W_{n,x} \leq X_n$. Therefore, if $X_n > \tilde{t}$ then
	\[
	Z_{n+1,x} = \begin{cases} 
	\displaystyle \frac{X_{n+1,x}}{W_{n,x}} \geq  \frac{X_{n+1,x}}{X_n}  ,  & \mbox{if} \quad W_{n,x} \geq t^{\ast},\\ 
	\displaystyle \frac{X_{n+1,x} - t^{\ast}}{W_{n,x}+t^{\ast}} + \frac{t^{\ast}}{\Tilde{t}} \geq \frac{X_{n+1,x} - t^{\ast}}{\Tilde{t}} + \frac{t^{\ast}}{\Tilde{t}} \geq  \frac{X_{n+1,x}}{X_n}, & \mbox{if} \quad  W_{n,x} < t^{\ast}.
	\end{cases}
	\]
	Therefore, $Z_{n+1} \geq X_{n+1}/X_n$ whenever $X_n > \Tilde{t}$. From whence we get that
	\begin{multline}\label{eq: conta}
	\displaystyle P\left( X_{n_0} > \Tilde{t}, ~ X_{n_0+1} \geq X_{n_0}, ~ \ldots, ~ X_{n_0+m} \geq X_{n_0}\right)\\
	= \displaystyle P\left( X_{n_0} > \Tilde{t}, ~ \frac{X_{n_0+1}}{X_{n_0}} \geq 1, ~ \ldots, ~ \prod_{l=1}^{m} \frac{X_{n_0+l}}{X_{n_0+l-1}} \geq 1\right)\\ 
	\leq \displaystyle P\left( X_{n_0} > \Tilde{t}, ~ Z_{n_0+1} \geq 1, ~ \ldots, ~ \prod_{l=1}^{m} Z_{n_0+l} \geq 1\right).
	\end{multline}
	
	Consider now the set $\Lambda = \left\{\gamma:=(j_1, \ldots, j_m) ~;~ \prod_{i=1}^{l} a_{j_i} \geq 1 ,~ \forall 1 \leq l \leq m \right\}$. We have that the last expression in \eqref{eq: conta} satisfies
	\begin{multline}\label{eq: conta2}
	\displaystyle P\left( X_{n_0} > \Tilde{t}, ~ Z_{n_0+1} \geq 1, ~ \ldots, ~ \prod_{l=1}^{m} Z_{n_0+m} \geq 1\right)\\ 
	\leq \displaystyle \sum_{\gamma \in \Lambda} P\left( X_{n_0} > \Tilde{t}, ~ Z_{n_0+1} \in I_{j_1}, ~ \ldots, ~ Z_{n_0+m} \in I_{j_m} \right)\\
	= \displaystyle \sum_{\gamma \in \Lambda} \left[ \prod_{l=1}^{m} P \left( Z_{n_0+l} \in I_{j_l} ~\big\vert~ A_l \right) \right],
	\end{multline}
	where $A_l = \{ X_{n_0} > \Tilde{t}, ~ Z_{n_0+1} \in I_{j_1}, ~ \ldots, ~ Z_{n_0+l-1} \in I_{j_{l-1}} \}$.
	
	To simplify the notation, we define the random vector $\xi: \Omega \rightarrow \realmathbb{R}^{(n_0-l-1)|V|}$, denoted by, $\xi = (\xi_{m,x})_{\{1 \leq m \leq n_0 + l -1,~ x \in V\}}, $ where $\xi_{m,x}(\omega) = W_{m,x}(\omega)$.
	Notice that
	$A_l$ is measurable in the $\sigma$-algebra generated by $\xi$.
	Let $\psi$ denote the function associating $\xi$ to $(X_{n_0}, ~ Z_{n_0+1}, ~ \ldots, ~ Z_{n_0+l-1})$,
	and make $\mathcal{R}_l=(\tilde t,\infty)\times I_{j_1}\times\ldots\times I_{j_{l-1}}$. 
	Let $\tilde F$ denote the distribution function of $\xi$.
	Thus, 
	\begin{eqnarray}\label{eq: Al}
	\nonumber P\left( \{ Z_{n_0+l} \in I_{j_l} \} \cap A_l \right) &=&  \int_{\psi^{-1}(\mathcal{R}_l)} P\left(Z_{n_0+l} \in I_{j_l}  ~|~ \xi=y \right) d\tilde F(y)\\
	&=& \int_{\psi^{-1}(\mathcal{R}_l)} P \left( \max_{x \in V} Z_{n_0+l,x} \in I_{j_l} ~ \bigg\vert ~ \xi = y \right) d\tilde F(y).
	\end{eqnarray}
	
	Since $A_l \subset \{X_{n_0+l-1} > \Tilde{t}\}$, using the Markov Property described in \eqref{eq: Markov} and the definition of $Z_{n+1}$ given in \eqref{eq: Z}, we obtain
	\begin{equation}\label{eq: V}
	P \left( \max_{x \in V} Z_{n_0+l,x} \in I_{j_l} ~ \bigg\vert ~ \xi = y \right)
	= P \Big( \max \big\{V_x(y) ~;~ x \in V, ~ x \neq x_{n_0+l-1} \big\} \in I_{j_l} \Big),
	\end{equation}
	where
	\[
	V_x(y) \sim \begin{cases} 
	\displaystyle \frac{E\big(y_{n_0+l-1,x}\big)}{y_{n_0+l-1,x}}, & \mbox{if } \quad y_{n_0+l-1,x} \geq t^{\ast},\\ 
	\displaystyle \frac{E\big(y_{n_0+l-1,x} + t^{\ast}\big)}{y_{n_0+l-1,x} + t^{\ast}} + \frac{t^{\ast}}{\Tilde{t}}, & \mbox{if } \quad  y_{n_0+l-1,x} < t^{\ast}.
	\end{cases}
	\]
	Recalling that  $\delta = t^{\ast}/\Tilde{t}$, note that the variables $V_x(\omega)$, $x \in V\setminus \{x_{n_0 + l -1}\}$, satisfy the conditions of Proposition  \ref{pr: dominacao} with $M = |V|-1$. Hence, for all $y \in \psi^{-1}(\mathcal{R}_l)$, we have 
	\[
	P \Big( \max \big\{V_x(y) ~;~ x \in V, ~ x \neq x_{n_0+l-1} \big\} \in I_{j_l} \Big) < Cp_{j_l}.
	\]
	Replacing this in \eqref{eq: Al} and \eqref{eq: V}, we get
	\[
	P\left( \{ Z_{n_0+l} \in I_{j_l} \} \cap A_l \right) \leq \int_{\psi^{-1}(\mathcal{R}_l)} Cp_{j_l} d\tilde F(y) 
	= Cp_{j_l} P(\xi\in\psi^{-1}(\mathcal{R}_l))= Cp_{j_l} P(A_l).
	\]
	Thus,  \eqref{eq: conta} and  \eqref{eq: conta2} yield
	\begin{multline}\label{eq: cota}
	\displaystyle P\left( X_{n_0} > \Tilde{t}, ~ X_{n_0+1} \geq X_{n_0}, ~ \ldots, ~ X_{n_0+m} \geq X_{n_0}\right)\\
	\leq \displaystyle \sum_{\gamma \in \Lambda} \left[ \prod_{l=1}^{m} P \left( Z_{n_0+l} \in I_{j_l} ~\big\vert~ A_l \right) \right]\\
	\leq \displaystyle \sum_{\gamma \in \Lambda} \left( \prod_{l=1}^{m} Cp_{j_l} \right) = C^{m} \displaystyle \sum_{\gamma \in \Lambda} \left( \prod_{l=1}^{m} p_{j_l} \right).
	\end{multline}
	
	Recalling the definition  of  $\Tilde{Y}$ in \eqref{eq: dominad}, since $\{\Tilde{Y}_ i \}_{i \in \realmathbb{N}}$ are i.i.d. with same distribution as $\Tilde{Y}$, we have that
	\begin{equation}\label{eq: cotafinal}
	\displaystyle \sum_{\gamma \in \Lambda} \left( \prod_{l=1}^{m} p_{j_l} \right) = \sum_{\gamma \in \Lambda}  \mathbb{P} \left( \Tilde{Y}_1 = a_{j_1}, \ldots, \Tilde{Y}_m= a_{j_m} \right) \leq \mathbb{P} \left( \prod_{l=1}^m \Tilde{Y_l} \geq 1 \right),
	\end{equation}
	and \eqref{eq: cota} and \eqref{eq: cotafinal} yield the proof.
\end{proof}

\subsection{Proof of item (1) of Theorem \ref{th: principal}}
\noindent Since $\mathbb{E}[\Tilde{Y}^{\theta}] \leq \mu/C$ --- see paragraph of \eqref{eq: dominad} --- , we have
\[
\mathbb{P} \left( \prod_{l=1}^m \Tilde{Y_l} \geq 1 \right) = \mathbb{P} \left( \prod_{l=1}^m \Tilde{Y_l}^{\theta} \geq 1 \right) \leq \mathbb{E}\left[ \Tilde{Y}^{\theta} \right]^{m} \leq \left(\frac{\mu}{C}\right)^{m}.
\]
Therefore, recalling that $\mu < 1$, it follows from the Proposition \ref{pr: domination} that 
\begin{multline*}
\displaystyle P\left( X_{n_0} > \Tilde{t}, ~ X_{n_0+l} \geq X_{n_0}, ~ \forall l \in \realmathbb{N} \right)\\
=      \displaystyle \lim_{m \rightarrow \infty} P\left( X_{n_0} > \Tilde{t}, ~ X_{n_0+1} \geq X_{n_0}, ~ \ldots, ~ X_{n_0+m} \geq X_{n_0}\right)\\ \leq \lim_{m \rightarrow \infty} \mu^{m} = 0.
\end{multline*}
Hence,
\[
\displaystyle P \left( \lim_{n \rightarrow \infty} X_n = \infty \right) \leq \displaystyle P \left( \bigcup_{n_0 \geq 1} \displaystyle \left\{ X_{n_0} > \Tilde{t}, ~ X_{n_0+l} \geq X_{n_0}, ~ \forall l \in \realmathbb{N} \right\} \right) = 0.
\]

It follows, as noted above --- see paragraph of \eqref{eq: continencia} ---,
that $P( \zeta_t \neq \emptyset, \forall t>0) = 0$ for every $\lambda > 0$.


\section{Appendix}

\begin{proposition}
	Let $Y$ be defined as in \eqref{eq: Y}, if $M \in \realmathbb{N}$ is such that  
	\[
	M < 1+ \frac{2\alpha -1}{(1-\alpha)(2-\alpha)},
	\]
	then $\mathbb{E}[\log(Y)] < 0$.
\end{proposition}
\begin{proof}
	Given $x > 0$, we have $P(Y \leq x) = P(Y_1 \leq x)^{M}$. Taking derivatives with respect to  $x$, we get that $Y$ has density $MC_{\alpha}^{M-1}g(x)^{M-1}f(x)$, where $f(x)$ is given by \eqref{eq: densidade} and
	\[
	g(x) = \begin{cases} 
	0, & \mbox{se } \quad x \leq 0,\\ 
	\int_{0}^{x} \frac{1}{t^{\alpha}(1+t)} dt, & \mbox{se } \quad x > 0.
	\end{cases}
	\]
	Hence, we have that
	\[
	\displaystyle \mathbb{E}[\log(Y)] = MC_{\alpha}^{M} \int_{0}^{\infty} \log(x) \frac{g(x)^{M-1}}{x^{\alpha}(x+1)} dx.
	\]
	Making the change of variables $u = 1/x$, we get
	\[
	\displaystyle  \int_{0}^{1} \log(x) \frac{g(x)^{M-1}}{x^{\alpha}(x+1)} dx = - \int_{1}^{\infty} \log(x) \frac{g(1/x)^{M-1}}{x^{1 -\alpha}(x+1)} dx.
	\]
	It readily follows that 
	\[
	\displaystyle \mathbb{E}[\log(Y)] = MC_{\alpha}^{M} \int_{1}^{\infty} \frac{\log(x)}{x+1} \left[ \frac{g(x)^{M-1}}{x^{\alpha}} - \frac{g(1/x)^{M-1}}{x^{1-\alpha}} \right] dx.
	\]
	It is sufficient to show that the term in brackets is negative whenever $x>1$. This is equivalent to 
	\[
	\frac{g(x)}{g(1/x)} \leq x^{\frac{2\alpha -1}{M-1}}, ~\forall x >1. 
	\]
	For simplicity, let $\beta = 1 -\alpha$. It follows from the hypothesis that, for $x>1$, 
	\[
	x^{\frac{2\alpha -1}{M-1}} \geq x^{(1-\alpha)(2-\alpha)} = x^{\beta(\beta +1)}.
	\]

	We define the auxiliary function $G: \realmathbb{R} \to \realmathbb{R}$, given by
	$G(x) = g(x) - x^{\beta(\beta +1)}g(1/x).$
	Thus, we have that $G(1) = 0$ and its derivative is 
	\begin{equation}\label{G'} 
	G'(x) = \frac{1 + x^{\alpha + \beta^2} - \beta(\beta +1) g(1/x) x^{\beta^2}(1+x)}{x^{\alpha}(1+x)}.
	\end{equation}
	
	Observe that, since for $0 < t \leq 1$ we have $1/(1+t) > 1 - t$, then, for all $x \geq 1$,
	\begin{eqnarray*}
		\nonumber   g(1/x) = \int_{0}^{1/x} \frac{1}{t^\alpha(1+t)}dt &>& \frac{1}{\beta x^{\beta}} - \frac{1}{(\beta +1) x^{\beta + 1}}\\ &=& \frac{x(\beta + 1) - \beta}{x^{\beta + 1}\beta(\beta +1)}.
	\end{eqnarray*}
	Hence $\beta(\beta +1)g(1/x) > [x(\beta +1) -\beta]/x^{\beta+1}$. To conclude that $G'(x) < 0$ for all $x > 1$, as stated in \eqref{G'}, it is enough to show that  
	\[
	\frac{x(\beta + 1) - \beta}{x^{\beta + 1}} > \frac{1+x^{\alpha + \beta^{2}}}{x^{\beta^2}(1+x)}, ~ \forall x > 1.
	\]
	Or equivalently, 
	\begin{multline*}
	x^{\beta^2}(1+x)x(\beta +1) - x^{\beta^2}(1+x)\beta - x^{\beta +1}(1+ x^{\alpha + \beta^2})\\ 
	=  x^{\beta^2}[(1+\beta)(x^2+x) - \beta(1+x) - (x^2 + x^{1+\beta - \beta^2})]\\
	=  x^{\beta^2}[x^2 + x + \beta(1+x)(x-1) - (x^2 + x^{1+\alpha\beta})]\\
	=  x^{\beta^2}[x+\beta(x+1)(x-1) - x^{1+\alpha\beta}] > 0, ~\forall x > 1,
	\end{multline*}
	Since the last inequality is true, we have $G'(x) < 0$ for all $x > 1$, and since $G(1)=0$, we conclude that $G(x) < 0$, for all $x >1$. The proof is finished.
\end{proof} 

\bigskip

\noindent{\bf Acknowledgements}
{L. R. Fontes acknowledges
	support of CNPq (grant 311257/2014-3), and FAPESP (grant 2017/10555-0). 
	P. A. Gomes acknowledges support of CAPES.
	R. Sanchis acknowledges support of
	CAPES, CNPq and
	FAPEMIG (Programa Pesquisador Mineiro).}

\bigskip




\vspace{.75cm}



\end{document}